\documentclass{amsart}

\usepackage{enumerate}
\usepackage{enumitem}
\usepackage{amsfonts}
\usepackage{xy}
\usepackage{sseq}
\input xy
\xyoption{all}
\usepackage{url}
\usepackage{verbatim}
\usepackage{hyperref}
\usepackage{xy}
\input xy
\xyoption{all}
\usepackage{amsthm}
\usepackage{amsmath}
\usepackage{amssymb}
\usepackage{cleveref}

\title{A thick subcategory theorem for modules over certain ring spectra}
\author{Akhil Mathew}
\email{amathew@math.berkeley.edu}
\address{Department of Mathematics \\ University of California \\  Berkeley, CA 94720}
\urladdr{http://math.berkeley.edu/~amathew}
\date{\today}

\newtheorem{lemma}{Lemma}[section]

\newtheorem{theorem}[lemma]{Theorem}
\newtheorem{corollary}[lemma]{Corollary}

\newtheorem{proposition}[lemma]{Proposition}
\theoremstyle{definition}
\newtheorem{definition}[lemma]{Definition}
 \theoremstyle{definition}
\newtheorem{remark}[lemma]{Remark}

\newcommand{\spec}{\mathrm{Spec}}

\newcommand{\im}{\mathrm{Im}}

\renewcommand{\sp}{\mathrm{Sp}}
\newcommand{\st}{\mathrm{Stack}}
\newcommand{\otop}{\mathcal{O}^{\mathrm{top}}}

\newcommand{\mell}{M_{{ell}}}

\renewcommand{\hom}{\mathrm{Hom}}
\renewcommand{\rightrightarrows}{\begin{smallmatrix} \to \\
\to \end{smallmatrix} }
\newcommand{\triplearrows}{\begin{smallmatrix} \to \\ \to \\ 
\to \end{smallmatrix} }

\newcommand{\TMF}{\mathrm{TMF}}
\renewcommand{\ell}{\mathrm{Ell}}

\newcommand{\md}{\mathrm{Mod}}
\newcommand{\qcoh}{\mathrm{QCoh}}
\newcommand{\pic}{\mathrm{Pic}}

\begin{document}

\begin{abstract}
We classify thick subcategories of the $\infty$-categories of perfect modules over ring spectra which arise as
functions on even periodic derived stacks satisfying affineness
and regularity conditions. For example, we show that the thick subcategories of perfect modules
over $\TMF$ are in natural bijection with the subsets of the
 underlying space of the moduli stack of elliptic curves which are closed
 under specialization.  
\end{abstract}

\maketitle

\section{Introduction}

\newcommand{\spp}{\sp_{(p)}^{\omega}}
\subsection{Generalities}
Let $\mathcal{C}$ be a stable $\infty$-category. Recall that: 

\begin{definition} 
A full subcategory
$\mathcal{C}' \subset \mathcal{C}$ is \emph{thick} if $\mathcal{C}'$ is a
stable subcategory (i.e., it is closed under finite limits and colimits), and
$\mathcal{C}'$ is closed under retracts.
\end{definition} 
Note that this definition only depends on the underlying \emph{homotopy category} of
$\mathcal{C}$ and its triangulated structure, and can be studied without the
language of $\infty$-categories. 
Since many properties of objects in
$\mathcal{C}$ are controlled by thick subcategories, it is generally very
useful to have classifications of the possible thick subcategories of
$\mathcal{C}$. 

In the setting of the stable $\infty$-category $\spp$ of finite $p$-local spectra, the following fundamental result
was proved by Hopkins and Smith: 
\begin{theorem}[Hopkins-Smith \cite{HS}] \label{HSthick}
There is a descending sequence of thick subcategories
\[ \spp = \mathcal{C}_0 \supsetneq \mathcal{C}_1 \supsetneq \mathcal{C}_2
\supsetneq
\dots, \]
such that every nonzero thick subcategory of 
$\spp$ is one of the $\mathcal{C}_i$. 
\end{theorem} 

The subcategories $\mathcal{C}_i$ can be described  in terms of the geometry of
the moduli stack $M_{FG}$ of formal groups. Namely, recall that any finite spectrum defines a
$\mathbb{Z}/2$-graded coherent sheaf on $M_{FG}$ (via its complex bordism). 
$M_{FG}$ localized at $p$ has a filtration by \emph{height}
\[ M_{FG} \supset M_{FG}^{\geq 1} \supset M_{FG}^{\geq 2} \supset \dots,  \]
such that the successive ``differences'' are quotients of a point by an
ind-\'etale group scheme (and in particular, this tower cannot be
refined); for a discussion of this, see \cite{goerssqc}.
 The moduli stack $M_{FG}^{\geq i}$ parametrizes formal groups of height at
 least $i$, and $M_{FG}^{\geq 1}$ parametrizes formal groups over a
 $\mathbb{F}_p$-algebra. A spectrum $X \in \spp$ belongs to $\mathcal{C}_i$ precisely when its
associated quasi-coherent sheaf is set-theoretically supported on $M_{FG}^{\geq
i}$. For instance, $\mathcal{C}_1$ consists of the $p$-torsion spectra. 

In stable homotopy theory, \Cref{HSthick} is extremely useful as a sort of
``Tauberian'' theorem. If one wishes to prove a property of all finite
($p$-local) spectra,
e.g., of the sphere, then it suffices to show that the property is thick and
that a single
finite spectrum with nontrivial rational homology satisfies it. 

After the results of \cite{HS}, thick subcategories have been studied in a number of other
settings. For instance, given a commutative ring $R$, 
one may consider the $\infty$-category $D_{\mathrm{perf}}(R)$ of perfect
complexes of $R$-modules, and one may consider the 
thick subcategories of $D_{\mathrm{perf}}(R)$. Given a subset $Z \subset \spec
R$ closed under specialization (equivalently, a union of closed subsets), one may define a thick subcategory of $D_{\mathrm{perf}}(R)$ consisting of
complexes whose cohomologies are set-theoretically supported on $Z$. 
When $R$ is noetherian, one has the following theorem: 

\begin{theorem}[Hopkins-Neeman \cite{global, neeman}] \label{HN}
The above construction establishes a bijection between
thick subcategories of $D_{\mathrm{perf}}(R)$ and subsets
of $\spec R$ closed under specialization. 
\end{theorem} 

Given a symmetric monoidal stable $\infty$-category $(\mathcal{C},
\otimes, \mathbf{1})$, 
one may also consider \emph{thick tensor ideals}: these are thick subcategories 
$\mathcal{D} \subset \mathcal{C}$ such that if $X \in \mathcal{D}, Y \in
\mathcal{C}$, then the 
tensor product $X \otimes Y$ belongs to $\mathcal{D}$ as well. Thick
tensor ideals have been extensively investigated \cite{Balmer}, and one
can, for instance, arrange the
``prime'' ones into
a topological space. 

For example, for the perfect derived category of a noetherian scheme, thick
tensor-ideals were classified in work of Thomason \cite{Thomason} generalizing
\Cref{HN}. 
In this case, again thick tensor-ideals are classified in terms of the
\emph{supports} of objects. The notion of ``support'' has been axiomatized in
work of Benson-Iyengar-Krause \cite{BIK} with a view towards diverse
applications, including stable module categories for finite groups, for
triangulated categories with an action of a commutative ring as endomorphisms
of the identity. 
In our setting, we will be working with symmetric monoidal
stable $\infty$-categories (such as $D_{\mathrm{perf}}(R)$), where the unit
object generates $\mathcal{C}$ under finite colimits, so thick subcategories are
automatically tensor ideals. 
Moreover, the ``base ring'' will be somewhat complicated, so it will be
convenient to work locally, using techniques of \cite{MM}.

\subsection{Methods}
The goal of this paper is to classify thick subcategories in a different setting:  that is, for
$\infty$-categories of perfect modules over structured ring spectra which
arise as global sections of the structure sheaf on  even periodic 
derived stacks. Such ring spectra play an important role in stable homotopy
theory: for instance, $\TMF$ arises as the ring of functions on a derived
version of the moduli stack of elliptic curves. 

Our goal is to understand the structure of the $\infty$-category
$\md^\omega(\TMF)$ of perfect $\TMF$-modules, for instance. One difficulty in
doing so is that the algebraic structure of $\pi_* \TMF$ is extremely
complicated, while the analysis of ring spectra and modules over them is greatly
simplified when one has nice (e.g., regular) homotopy rings. Nonetheless, we
know that $\TMF$ is obtained as the homotopy inverse limit of a diagram of
\emph{elliptic spectra}, 
which are even periodic $E_\infty$-rings whose formal group is associated to an
elliptic curve classified by an \'etale map to $M_{ell}$. 
More precisely, there is a \emph{derived Deligne-Mumford stack} $(M_{ell}, \otop)$, whose
underlying ordinary stack is the moduli stack of elliptic curves, such that the
global sections of the structure sheaf $\otop$ gives $\TMF$. 
These elliptic spectra obtained by evaluating $\otop$ on an affine scheme
\'etale over $M_{ell}$ are much better behaved: not only are they even
periodic, but their homotopy rings are regular noetherian. The classification
of thick subcategories for perfect modules over \emph{them} is significantly
simpler by the following result which will be proved in \Cref{thickaffine} below. 

\begin{theorem} 
Suppose $R$ is an even periodic $E_{\infty}$-ring with $\pi_0(R)$ regular
noetherian. Then there is a canonical bijection between thick subcategories of
the $\infty$-category of perfect $R$-modules and subsets of $\spec \pi_0(R)$
closed under specialization.
\end{theorem} 

Therefore, we shall approach the classification of thick subcategories of
$\md^\omega(\TMF)$ by relating $\md^\omega( \TMF)$ to the $\infty$-categories
of perfect modules 
over elliptic spectra. The essential ingredients for doing this are in
\cite{MM}. In that paper, L. Meier and the author prove:

\begin{theorem}[\cite{MM}] The $\infty$-category of
$\TMF$-modules $\md(\TMF)$ is equivalent the $\infty$-category of quasi-coherent sheaves on the derived
moduli stack of elliptic curves.
\end{theorem}
We refer to \cite[\S 2.3]{DAGQC} for
generalities on quasi-coherent sheaves on derived stacks. By \cite[Prop.
2.3.12]{DAGQC}, 
we can rewrite this statement by saying that 
\[ \md( \TMF) \simeq \varprojlim_{\spec R \to M_{ell}}  \md(\otop( \spec R)), \]
as $\spec R \to M_{ell}$ ranges over the \'etale maps from affine schemes.

More generally, the main result of \cite{MM} gives a criterion for this
phenomenon of 
``affineness,'' which had been first explored by Meier in \cite{meier}.

\newcommand{\affx}{\mathrm{Aff}_{/X}^{\mathrm{et}}}
\newcommand{\CAlg}{\mathrm{CAlg}}
We briefly review the setup. Let $X$ be a Deligne-Mumford stack equipped with a
flat map $X \to M_{FG}$. Given any affine scheme $\spec R$ and an \'etale map
$\spec R \to X$, the composite $\spec R \to X \to M_{FG}$ classifies a formal
group over $R$ which yields an associated (weakly) even periodic,  Landweber-exact homology theory and
even a homotopy commutative ring spectrum. 
In particular, we get a functor
\[ \affx \to \text{Homotopy commutative ring spectra}  \]
from the category of affine schemes \'etale over $X$ to the category of
homotopy commutative ring spectra and homotopy classes of maps between them. 
We refer to \cite[Lecture 18]{lurienotes} for the fundamentals of  even periodic 
Landweber-exact, ring spectra.

In certain important cases, one has a lifting
\[ \xymatrix{
  & \ar[d] \CAlg  \\
\affx \ar[ru]^{\otop} \ar[r] &  \text{Homotopy commutative ring spectra} 
} ,\]
where $\CAlg$ is the $\infty$-category of $E_\infty$-rings. 
Such a lift is called an \emph{even periodic refinement} $\mathfrak{X} = (X,
\otop)$ of $X \to M_{FG}$. Such even periodic refinements exist, for instance,
for $X$ the moduli stack of elliptic curves, and yield important examples of
derived stacks. 

Given an even periodic refinement $\mathfrak{X} = (X, \otop)$, we can consider
the $E_\infty$-ring $\Gamma(\mathfrak{X}, \otop)$ of ``functions'' on the
derived stack $\mathfrak{X}$ and its $\infty$-category of modules. We can also
consider a related $\infty$-category, the $\infty$-category $\qcoh(
\mathfrak{X})$ of \emph{quasi-coherent sheaves} on $\mathfrak{X}$, obtained (by
\cite[Prop.
2.3.12]{DAGQC})
as
the homotopy limit of the module categories $\md( \otop( \spec R))$ as $\spec R
\to X$ ranges over all \'etale maps from affines. 

The main result of \cite{MM} now runs as follows. 
\begin{theorem}[\cite{MM}] 
\label{affinethm}
Let $\mathfrak{X} = (X, \otop)$ be an even periodic refinement of a
Deligne-Mumford stack $X$ equipped with a flat map $X \to M_{FG}$. If $X \to
M_{FG}$ is quasi-affine, then the global sections functor establishes an
equivalence of symmetric monoidal $\infty$-categories
\[ \Gamma\colon \qcoh( \mathfrak{X}) \simeq \md( \Gamma( \mathfrak{X}, \otop)).  \]
\end{theorem} 

The hypotheses of \Cref{affinethm} apply in several important cases, such as
the derived version of the moduli stack of elliptic curves as well as its
Deligne-Mumford compactification.

\subsection{Results}
In this paper, we will use \Cref{affinethm} to illuminate the structure of the
$\infty$-category of \emph{perfect} $\Gamma( \mathfrak{X}, \otop)$-modules. 
In particular, we will classify the thick subcategories of perfect $\Gamma(
\mathfrak{X}, \otop)$-modules, under some further constraints.

In classifying thick subcategories, it will generally be more convenient to work with $\qcoh( \mathfrak{X})$ (at
least implicitly), and in particular with the \emph{homotopy group sheaves}
$\pi_0, \pi_1$ of
a given object in $\qcoh( \mathfrak{X})$, which in the perfect case will define
a $\mathbb{Z}/2$-graded coherent sheaf on $X$. Given a subset of the
underlying space of $X$ closed under specialization, it
follows that we can define a thick subcategory of $\qcoh( \mathfrak{X})$
consisting of objects whose homotopy group sheaves are supported on that subset. 

The main result of this paper is: 

\begin{theorem} \label{ourthick}
Suppose $X$ is regular and affine flat over $M_{FG}$. 
Then the  thick subcategories of perfect modules over $\md( \Gamma(
\mathfrak{X}, \otop)) \simeq \qcoh(\mathfrak{X})$ are in bijection (as indicated above) with the
subsets of $X$ closed under specialization. 
\end{theorem}

\Cref{ourthick} in particular applies to the derived moduli stack of elliptic
curves, and one has:

\begin{corollary} 
There is a bijection 
between thick subcategories of perfect $\TMF$-modules and subsets of $M_{ell}$
which are closed under specialization.
\end{corollary} 

We can also apply this to the classification of thick subcategories of perfect modules over the
Hopkins-Miller $EO_n$-spectra (in view of \cite[\S 6.2]{MM}). 
\begin{corollary} 
Let $G$ be a finite subgroup of the $n$th Morava stabilizer group, so that $G$
acts on Morava $E$-theory $E_n$. Then thick subcategories of perfect
$E_n^{hG}$-modules are in bijection with $G$-invariant subsets of $\spec \pi_0
E_n$ closed under specialization. 
\end{corollary}

The proof of \Cref{ourthick} follows the outline of \Cref{HSthick}, although
since we already have the nilpotence technology of \cite{DHS} and \cite{HS}
and can in particular use results such as \Cref{HSthick},
many steps are much simpler. The proof that all possible thick subcategories
come from subsets of $X$ closed under specialization is essentially formal once certain ``residue
fields'' are constructed, using the techniques of \cite{BakerRichter} and
\cite{BR2}. (The analog in \Cref{HSthick} is given by the Morava $K$-theories.) 
The harder step is to show that all the different subsets closed under
specialization are realized,
which requires in addition some algebraic preliminaries about the structure
of $M_{FG}$ and topological preliminaries of vanishing lines in spectral
sequences, which are of interest in themselves. 

\subsection*{Acknowledgments} I would like to thank Benjamin Antieau, Mike Hopkins, Jacob
Lurie, Lennart
Meier, and Niko Naumann for helpful discussions related to the subject of this
paper, and the referee for many detailed comments. 
The author was partially supported by the NSF Graduate Research Fellowship
under grant DGE-110640.
\newcommand{\mo}{\md^{\omega}}
\newcommand{\supp}{\mathrm{Supp}}
\section{The affine case}
\label{sec:affinecase}

\subsection{Setup}
Let $R$  be an  even periodic $E_\infty$-ring such that $\pi_0
R$ is \emph{regular noetherian}. We do not need to assume that $R$ is Landweber-exact.
Consider the
stable $\infty$-category $\mo(R)$ of perfect $R$-modules. The goal of this
section is to classify the thick subcategories of $\mo(R)$ in
\Cref{thickaffine} below. 

\begin{proposition} 
Let $R$ be as above, and 
let $M$ be a perfect $R$-module. Then:
\begin{enumerate}
\item  $\pi_0 M \oplus \pi_1 M$ is a
finitely generated $\pi_0 R$-module.
\item Assume that $\pi_0R$ has finite global (i.e., Krull) dimension. 
In this case, an $R$-module $M$ is in fact perfect \emph{if and only if} $\pi_0 M \oplus
\pi_1 M$ is a finitely generated $R$-module.
\end{enumerate}
\end{proposition} 
\begin{proof}
The first assertion is equivalent to the assertion that $\pi_*(M)$ is a
finitely generated $\pi_*(R)$-module. Consider the collection of $M \in \mo(R)$ for which
this holds. It clearly contains $R$. Since $\pi_*(R)$ is noetherian, it follows
from the long exact sequence of a cofibering
that this collection is a stable subcategory of $\mo(R)$, and it is also closed
under retracts. Therefore, it is all of $\mo(R)$.

Suppose conversely that $\pi_0(R)$ has finite global dimension and $\pi_*(M)$
is a finitely generated 
$\pi_*(R)$-module. We claim that $M$ is perfect. For this, we use descending
induction on the (finite) projective dimension of $\pi_0(M) \oplus \pi_1(M)$ over $\pi_0(R)$.
If $\pi_0(M), \pi_1(M)$ are projective $R$-modules, then it is well-known that
$M$ is perfect: in fact, $M$ is a retract of a sum of copies of $R$ and $\Sigma
R$. Suppose the projective dimension of $\pi_0(M) \oplus \pi_1(M)$ is equal to
$n > 0$. Choose a map $R^s \vee \Sigma R^t \to M$ which induces an epimorphism
on $\pi_*$, which we may do as $\pi_*(M)$ is a finitely generated
$\pi_*(R)$-module, and let $F$ be the fiber of this map. Clearly, $M \in
\mo(R)$ if and only if $F \in \mo(R)$. But $\pi_*(F)$ is finitely generated,
and $\pi_0(F) \oplus \pi_1(F)$ has projective dimension $\leq n-1$. By
induction, we are done.
\end{proof}

\begin{definition} 
If $M \in \mo(R)$, the \emph{support} of $\pi_0 M \oplus \pi_1 M$ thus defines a
closed subset of $\spec \pi_0 R$, which we will simply write as $\supp M$. 
\end{definition} 
Given a cofiber sequence
\[ M' \to M \to M'' \to M'[1],  \]
we have 
\[  \supp M \subset \supp M' \cup \supp M'', \]
and, furthermore, the support only shrinks under taking retracts. 
This enables us to define thick subcategories of $\mo(R)$:

\begin{definition} 
Given a closed subset $Z \subset \spec \pi_0 R$, we let $\mo_Z(R) \subset
\mo(R)$ be the full subcategory of perfect $R$-modules $M$ such that $\supp M
\subset Z$. 
It follows from the preceding paragraph that $\mo_Z(R)$ is a thick subcategory
of $\md(R)$. 
\end{definition}

In general, we cannot expect every thick subcategory of $\mo(R)$ to come from a
closed subset $Z \subset R$. We might choose instead to work with
a \emph{family} of possible choices of $Z$. If $R$ is a domain, we could, for example,  consider
the collection of all $M \in \mo(R)$ which are \emph{torsion}: i.e., those
whose support does not contain the generic point. This is a thick subcategory,  
but it is not associated to any single closed subset, but rather to the
collection of all closed subsets of $\spec R$ that do not contain the generic
point. 
Rather than working with ``families of closed subsets,'' it is simpler to work
with subsets of $\spec R$ closed under specialization. 

\begin{definition} 
Let $Z \subset \spec \pi_0 R$ be a subset of $\spec \pi_0 R$ closed under specialization.
In this case, we define $\mo_Z(R) \subset \mo(R)$ analogously to be the
collection of $M \in \mo(R)$ with $\supp M \subset Z$. 
We thus get a \emph{map} from the collection of specialization-closed subsets
of $\spec \pi_0 R$ to the collection of thick subcategories of $\mo(R)$.
\end{definition}

The goal of this section is to prove 
that the $\mo_Z(R)$ are \emph{precisely} the thick subcategories of $\mo(R)$,
as $Z$ ranges over the specialization-closed subsets of $\spec \pi_0 R$.
The proof will follow the argument for finite spectra in \cite{HS}. The key
step, as in \cite{HS}, is a version of the nilpotence theorem, which is much
easier in the present setting. 

\subsection{Residue fields}
The first step is to define ``residue fields.''  Consider the regular ring $\pi_0 R$. For each
prime ideal $\mathfrak{p} \subset \pi_0 R$, the localization
$(\pi_0 R)_{\mathfrak{p}}$ is a regular local ring, whose maximal ideal $\mathfrak{p}
(\pi_0 R)_{\mathfrak{p}}$ is generated by a system of parameters $x_1, \dots, x_n \in
(\pi_0 R)_{\mathfrak{p}}$ such that $(\pi_0 R)_{\mathfrak{p}}/(x_1, \dots, x_n)$ is a field
$k(\mathfrak{p})$.  

\begin{definition}
\label{k(p)}
Given $\mathfrak{p}$, consider 
first $R_{\mathfrak{p}}$, the arithmetic
localization of $R$ at $\mathfrak{p}$, as an $E_{\infty}$-$R$-algebra using 
\cite[Theorem 2.2]{EKMM} or \cite[\S 8.2.4]{higheralg}.
Then consider the $R$-module $K(
\mathfrak{p})$ defined as
\[ K( \mathfrak{p}) = R_{\mathfrak{p}}/(x_1, \dots, x_n)
\stackrel{\mathrm{def}}{=} R_{\mathfrak{p}}/x_1 \wedge_R
R_{\mathfrak{p}}/x_2 \wedge_R \dots \wedge_R R_{\mathfrak{p}}/x_n
,  \]
where $R/x$ for $x \in \pi_0 R$ denotes the cofiber of $x\colon R \to R$,
so that 
\[ K( \mathfrak{p})_* \simeq k( \mathfrak{p})[t^{\pm 1}], \quad |t| = 2.  \]
By the results of \cite{angeltveit}, it follows that $K(\mathfrak{p})$ admits the structure of
an $A_\infty$-algebra internal to $\md(R)$. \end{definition}

\begin{definition}
For any $R$-module $M$, we can form the new
$R$-module
$ K( \mathfrak{p}) \wedge_R M$ 
and in particular we obtain (by applying $\pi_*$) a \emph{homology theory} $K( \mathfrak{p})_*$ on
the category of $R$-modules, taking value in the category of graded $k(
\mathfrak{p})[t^{\pm 1}]$-modules. 
\end{definition}
This homology theory 
$K(\mathfrak{p})$
is multiplicative and satisfies a K\"unneth isomorphism. Strictly speaking, the notation is
abusive, since the multiplicative structure $K(\mathfrak{p})_*$ was not constructed only from $\mathfrak{p}$, but
also used some other choices; any of those choices will be fine for our
purposes. 

The thick subcategories $\md_Z^{\omega}(R)$ can also be defined in terms of the
homology theories $K( \mathfrak{p})_*$. 

\begin{proposition} 
Given a \emph{perfect} $R$-module $M$, we have that $M \in \md_Z^\omega(R)$ if
and only if for every $\mathfrak{p} \notin Z$, 
$K( \mathfrak{p})_* M = 0$.
\end{proposition} 
\begin{proof}
It suffices to show that $M_{\mathfrak{p}}  = 0$ if and only if
$K(\mathfrak{p})_* M = 0$, for any $M \in \mo(R)$ and $\mathfrak{p} \in \spec
\pi_0 R$.
This is a consequence of the fact that if
$M_{\mathfrak{p}}/(x_1, \dots, x_n)$ is contractible and $M$ is perfect,  so is 
the Bousfield localization of $M$ at $K(\mathfrak{p})$, i.e., the smash product 
$M \wedge_R \widehat{R_{\mathfrak{p}}}$ where 
the completion  $\widehat{R_{\mathfrak{p}}}$ is given by
\[ \widehat{R_{\mathfrak{p}}} \simeq \varprojlim_N R_{\mathfrak{p}}/(x_1^N,
x_2^N, \dots, x_n^N) . \]
We refer to \cite[\S 4]{DAGXII} for generalities on completions of
$E_\infty$-rings and modules. 
By classical commutative algebra, $\widehat{R_{\mathfrak{p}}}$ is faithfully flat over $R_{\mathfrak{p}}$, so
that $M_{\mathfrak{p}} $ is itself contractible. 
\end{proof}

The basic step towards the thick subcategory theorem is given by the following Bousfield decomposition in $\md(R)$:
\begin{proposition} \label{bdecmdR}
If $M$ is an $R$-module, then $M \simeq 0$ if and only if $K( \mathfrak{p})_*
M  = 0$ for all $\mathfrak{p} \in \spec A$. 
\end{proposition} 
Note that we make no compactness assumptions on $M$ in this proposition. 
\begin{proof} 
One direction is obvious, so suppose $M$ is acyclic with respect to all the
homology theories $K( \mathfrak{p})_*$. We would like to show that $M$ is
contractible. Suppose the contrary. 

Since a module over $\pi_0 R$ vanishes if and only if all its localizations
at prime ideals vanish, we can assume that $\pi_0 R$ is a (regular) local ring. We
may use induction on the dimension of $\pi_0 R$, and thus assume that the
localization of $M$ at any non-maximal prime ideal $\mathfrak{p} \subset \pi_0 R$
 is trivial. 
It follows that if $x $ belongs to the maximal ideal $\mathfrak{m} \subset
\pi_0 R$,
then the $R[x^{-1}]$-module $M[x^{-1}]$ has the property that its localization
at any prime ideal is trivial, so $M[x^{-1}] \simeq 0$ for any $x \in
\mathfrak{m}$. 

Now we use a lemma from \cite[Lemma 1.34]{ravloc}: if $N$ is a nontrivial $R$-module,
then, for any $a \in \pi_0 R$, at least one of $N[x^{-1}]$ and $N/xN$ has to be
nontrivial. (This assertion would be false in
ordinary algebra.) It follows that $M/xM$ for each $x \in \mathfrak{m}$ is
nontrivial. Repeating this, and applying the same argument, it follows that
$M/(x_1, \dots, x_n) M$ is nontrivial for any system of parameters $(x_1, \dots,
x_n) $ for $\mathfrak{m}$. It follows that $K(\mathfrak{m})_* M \neq 0$. 
\end{proof}

Given the Bousfield decomposition, the rest of the proof of \Cref{thickaffine}
(which follows \cite{HS}) is now
completely formal, except for the last statement, and has been axiomatized in
\cite[\S 6]{axiomatic}. For completeness, we give a quick review of the argument. 

\begin{corollary} \label{fieldnilpotent}
If $\phi\colon \Sigma^k M \to M$ is a self-map in $\mo(R)$, then $\phi$ is nilpotent
if and only if $K(\mathfrak{p})_* (\phi)$ is nilpotent for each $\mathfrak{p}
\in \spec \pi_0 R$. 
\end{corollary} 
\begin{proof} This is a formal consequence of the Bousfield decomposition. 
Namely, since $M$ is compact, it follows that $\phi$ is nilpotent if and only if the
colimit 
of
\[ M \stackrel{\phi}{\to} \Sigma^{-k}M \stackrel{\phi}{\to} \Sigma^{-2k}M \to
\dots \]
is contractible, which by \Cref{bdecmdR} holds if and only if, for each $\mathfrak{p}
$, the diagram of finitely generated $k( \mathfrak{p})[t^{\pm 1}]$-modules induced by applying $K(
\mathfrak{p})_*$ has zero colimit. This implies the result. 
\end{proof} 

\begin{corollary} \label{nilpR}
Let $R'$ be an $R$-ring spectrum: that is, 
a monoid object in the homotopy category of $R$-modules. 

\begin{enumerate}
\item 
Suppose $\alpha \in \pi_* R$ maps to a nilpotent element under  
the Hurewicz map $\pi_* R \to K( \mathfrak{p})_* R$ for each $\mathfrak{p}$;
then $\alpha$ is nilpotent. 
\item Suppose a class $\alpha\colon R \to R' \wedge_R F$, for $F$ an 
$R$-module, is zero in $K(\mathfrak{p})_*$ for each $\mathfrak{p}$. Then it 
has the property that $\alpha^k\colon R \to R' \wedge_R F^{\wedge k}$ is null for
$k \gg 0$. 
\end{enumerate}
\end{corollary} 
\begin{proof} 
The first claim follows using a similar telescope; it implies the second claim
by considering $R' \wedge JF$ for $JF = R \oplus F \oplus F^{\wedge 2} \oplus
\dots $ the free $A_\infty$-algebra (in $\md(R)$) on $F$. 
\end{proof}

\subsection{Proof of the main result}
We are now ready to prove a thick subcategory theorem in the affine case. 
We will split this into two pieces. 

\begin{proposition} 
\label{firstclaim}
Let $M, N \in \mo(R)$. Suppose $\supp N \subset \supp M$. Then the thick
subcategory generated by $M$ contains $N$. 
\end{proposition} 
\begin{proof}The first claim is a consequence of \Cref{fieldnilpotent}. 
Hypotheses as in the claim, we consider a fiber sequence
\[ F \stackrel{\phi}{\to}R \to M \wedge_R \mathbb{D}M,  \]
where $\mathbb{D}M = \hom_R(M, R)$ is the Spanier-Whitehead dual (in
$R$-modules). By hypothesis, $\phi$ has the property that it 
is the zero map in $K( \mathfrak{p})_*$-homology for $\mathfrak{p} \in \supp M$ (but
not otherwise). 

Moreover, the cofiber of $\phi$ belongs to the thick subcategory $\mathcal{C}_M
\subset \mo(R)$ generated by $M$. 
Thus the cofiber of each iterated tensor power $\phi^{\wedge r}\colon F^{\wedge r}
\to R$ belongs to $\mathcal{C}_M$ too, and similarly for the cofiber of 
$$1_N \wedge
\phi^{r}\colon N \wedge F^{\wedge r} \to N.$$
But for large $r \gg 0$, these maps are zero. 
In fact, adjointing over, $\phi$ gives a map $\psi\colon R \to \mathrm{End}(N) \wedge
\mathbb{D}F$ and we equivalently need to show that the sufficiently high
iterated tensor powers $\psi^r\colon R \to \mathrm{End}(N) \wedge
\mathbb{D}F^{\wedge r}$
are zero. Since $\psi$ is zero on $K( \mathfrak{p})_*$-homology for
\emph{all} $\mathfrak{p}$, this follows from \Cref{nilpR}. 

If $1_N \wedge \phi^r$ is nullhomotopic for $r \gg 0$, it follows that the
cofiber (which we saw belongs to $\mathcal{C}_M$) contains $N$ as a direct
summand, and we are done with the first claim. 
\end{proof}

\begin{proposition} 
\label{secondclaim}
Let $Z \subset \spec \pi_0 R$. Then there exists $M \in \mo(R)$ such that
$\supp M = Z$.
\end{proposition} 
\begin{proof}
 Let $Z \subset \spec \pi_0 R$ be a closed subset corresponding
to the radical ideal $I$. Choose generators $x_1, \dots, x_n \in I$ and take as
the desired module 
\[ M = R/(x_1, \dots, x_n) \simeq R/x_1 R \wedge_R \dots \wedge_R R/x_n R.  \]
Clearly this vanishes when any of the $x_i$ are inverted. Conversely, if
$\mathfrak{p} \supset I$, then $R/x_i R$ has nontrivial $K(
\mathfrak{p})_*$-homology (since multiplication by $x_i$ induces the zero map
on $K( \mathfrak{p})_*$), 
and therefore the tensor power that defines $M$ has nontrivial $K(
\mathfrak{p})_*$-homology as well. 
\end{proof}

\begin{theorem} \label{thickaffine}
The thick subcategories of $\mo(R)$ are precisely the
$\left\{\mo_Z(R)\right\}$ for $Z \subset \spec \pi_0 R$ closed under
specialization, and these are all distinct. 
\end{theorem} 

We note that this result has been independently obtained in forthcoming work
of Benjamin
Antieau, Tobias Barthel, and David Gepner, and is likely known to others as
well. 

\begin{proof} 
Given a thick subcategory
$\mathcal{C} \subset \mo(R)$,
we associate to it the union $Z'$ (in $\spec R$) of the supports of all the
objects in $\mathcal{C}$. 
Note that if $Z_0, Z_1$  occur as supports of objects in $\mathcal{C}$, then so does
$Z_0 \cup Z_1$, by taking the direct sum. 
\Cref{firstclaim} will imply that this subset $Z'$, which is closed under specialization,
determines $\mathcal{C}$ in turn. 
\Cref{secondclaim} will imply that we can obtain any subset $Z' \subset \spec R$ closed under
specialization in this manner, by taking these modules for each closed subset
of $Z'$ and the thick subcategory they (together) generate.

\end{proof}

\section{Vanishing lines in the descent spectral sequence}

In the previous section, we classified the thick subcategories of
$\md^\omega(R)$ when $R$ is an even periodic $E_\infty$-ring with $\pi_0 R$
regular noetherian. 
The classification relied on the construction of certain ``residue fields'' of
$R$ to detect nilpotence, and then the verification that all the possible
thick subcategories one could thus hope for actually exist. The latter part was
very easy in the affine case, but it is somewhat trickier in the general case:
we actually will need to produce elements in the homotopy groups of $\Gamma(
\mathfrak{X}, \otop)$ for $\mathfrak{X} = (X, \otop)$ an even periodic derived
stack.

We will do this by producing such elements in the descent spectral sequence, and
by
establishing a general horizontal vanishing line for such descent spectral sequences
(based on nilpotence technology). The latter is the goal of this section, and
may be of independent interest. We
note that in the setting of \Cref{HSthick}, one is not working  in an
$E_n$-localized setting, so the corresponding vanishing line arguments are
significantly more difficult than they are for us. 

\subsection{Towers}
Let $\mathcal{C}$ be a stable $\infty$-category. 

\begin{definition} 
Let $\mathrm{Tow}( \mathcal{C})$ be the $\infty$-category of \emph{towers} of
objects of $\mathcal{C}$: that is $\mathrm{Tow}( \mathcal{C}) \simeq
\mathrm{Fun}( (\mathbb{Z}_{\geq 0})^{op},
\mathcal{C})$.
\end{definition} 

 To any element of $\mathrm{Tow}( \mathcal{C})$ and object $P \in
\mathcal{C}$, there is associated a spectral
sequence converging to the homotopy groups of the space (or spectrum) of maps from $P$ into
the homotopy inverse limit. For instance, given a cosimplicial spectrum, the
spectral sequence associated to the $\mathrm{Tot}$ tower is the Bousfield-Kan
homotopy spectral sequence. 

\newcommand{\tow}{\mathrm{Tow}}
\newcommand{\towg}{\mathrm{Tow}^{\mathrm{fast}}}
\newcommand{\townil}{\mathrm{Tow}^{\mathrm{nil}}}
Let us single out a certain subcategory $\mathrm{Tow}^{\mathrm{nil}}(
\mathcal{C})$ of $\mathrm{Tow}( \mathcal{C})$ consisting of objects whose
spectral sequences have extremely good convergence properties.
\begin{definition}
$\townil_{}(\mathcal{C}) \subset \tow(\mathcal{C})$ is the 
subcategory of $\mathrm{Tow}( \mathcal{C})$ consisting of towers $\dots \to X_n
\to X_{n-1} \to \dots \to X_0$ with the property 
that there exists an integer $r > 0$, such that all $r$-fold composites in the
tower are nullhomotopic. The subcategory $\townil (\mathcal{C})$ is the union of an ascending sequence
\[ \townil_1 (\mathcal{C})\subset \townil_2 (\mathcal{C})\subset \dots ,  \]
where $\townil_{r'} \subset \townil$ consists of towers such that every
$r'$-fold composite is null. 
\end{definition}

\begin{definition}
More generally, given a collection of objects $\mathfrak{U} \subset
\mathcal{C}$, we define $\townil_{\mathfrak{U}}( \mathcal{C})$ as the
collection of towers $\dots \to X_n \to X_{n-1} \to \dots $ such that there
exists $r \in \mathbb{Z}_{>0}$ such that every $r$-fold composite $X_n \to
X_{n-r}$ induces the zero map
\[ [U, X_{n}]_* \to [U, X_{n-r}]_*, \quad \text{for each } U \in \mathfrak{U}.  \]
Taking $\mathfrak{U} = \mathcal{C}$ gives $\townil( \mathcal{C})$. 
\end{definition}

We also need a generalization of $\townil_{\mathfrak{U}}( \mathcal{C})$ to
handle towers that are ``very close to constant,'' but not necessarily at zero. 
\begin{definition}
We define $\towg_{\mathfrak{U}}( \mathcal{C})$ to be the collection of those towers 
$\left\{X_n\right\}$ such that 
$\varprojlim_i X_i$ exists and such that the cofiber of the map  of towers
\[ \left\{\varprojlim_{i } X_i\right\}  \to \left\{X_n\right\}, \]
where the first term is the constant tower, belongs to $\townil_{\mathfrak{U}}(
\mathcal{C})$.  
\end{definition}

The basic permanency property of these subcategories is given by:

\begin{proposition}[Hopkins-Palmieri-Smith \cite{HPS}] \label{hpsprop}
For each $\mathfrak{U} \subset \mathcal{C}$, 
$\townil_{\mathfrak{U}}( \mathcal{C}) ,
\towg_{\mathfrak{U}}(\mathcal{C})\subset \mathrm{Tow}( \mathcal{C})$ are thick
subcategories. 
\end{proposition} 

\begin{proof} 
The assertion about 
$\towg_{\mathfrak{U}}(\mathcal{C})$ follows from the assertion about 
$\townil_{\mathfrak{U}}( \mathcal{C}) $, so we need only handle this case.
\Cref{hpsprop} is {essentially} contained in \cite[Cor. 2.3]{HPS}, but we give
the proof. It is easy to see that $\townil_{\mathfrak{U}}(\mathcal{C}) \subset
\tow(\mathcal{C})$ is closed under retracts and under suspensions
and desuspensions. It
suffices to show 
that given a cofiber sequence in $\tow(\mathcal{C})$
\[ \left\{X_n\right\}  \to \left\{Y_n\right\} \to \left\{Z_n\right\},\]
if $\left\{X_n\right\}, \left\{Z_n\right\} \in
\townil_{\mathfrak{U}}(\mathcal{C})$, then $\left\{Y_n\right\} \in
\townil_{\mathfrak{U}}(\mathcal{C})$. 

To see this, suppose every $r$-fold composite in $\left\{X_n\right\}$ induces
the zero map on $[U, \cdot]_*$ for each $U \in \mathfrak{U}$, and every $r'$-fold composite in
$\left\{Z_n\right\}$ induces the zero map on $[U, \cdot]_*$ for each $U \in
\mathfrak{U}$. We claim that every $(r + r')$-fold composite  in
$\left\{Y_n\right\}$ induces the zero map on 
$[U, \cdot]_*$ for each $U \in \mathfrak{U}$. Fix $U \in \mathfrak{U}$ and let
$n \geq 0$. Choose any map $f \colon \Sigma^k U \to Y_{n+r+r'}$ for $k \in \mathbb{Z}$
arbitrary.  Then we have a commutative diagram 
in $\mathcal{C}$,
\[ \xymatrix{
X_{n+r + r'} \ar[d] \ar[r] & Y_{n+r + r'}\ar[d] \ar[r] &  Z_{n+r + r'}\ar[d]  \\
X_{n+ r} \ar[d] \ar[r] & Y_{n+ r}\ar[d] \ar[r] &  Z_{n+ r}\ar[d] \\
X_n \ar[r] &  Y_n \ar[r] &  Z_n
}.\]
A diagram chase now shows that the composite of $f$ with $Y_{n+r + r'} \to Y_n$
is null. In fact, the composite of $f$ with $Y_{n+r + r'} \to Y_{n+r} \to
Z_{n+r}$ is null, and so factors through $X_{n+r}$. But the composite of any
map from $\Sigma^k U \to X_{n+r}$ with $X_{n+r} \to X_n$ is null. 
\end{proof}

\begin{remark}
Given a tower $\dots \to X_n \to X_{n-1}\to \dots \to X_0$ in
$\tow(\mathcal{C})$, it can belong to
$\townil(\mathcal{C})$ only if it is \emph{pro-zero}: that is, if the associated pro-object 
 in $\mathrm{Pro}(\mathcal{C})$ is equivalent to the zero pro-object, as its
 inverse limit is contractible and this property is preserved under any exact functor. 
The converse  is false: if a cofinal subset of the $X_i$'s are contractible,
the tower is automatically pro-zero, but such a tower need not belong to
$\townil$. The associated spectral sequence may support arbitrarily long
differentials. 
\end{remark}

\subsection{Vanishing lines}
We keep the notation of the previous subsection. 
We can give an interpretation of $\townil_{\mathfrak{U}}(\mathcal{C})$
in terms of vanishing lines. 
Before it, we make the following definitions. 

\begin{definition} 
Given an inverse system $\dots \to A_n \to A_{n-1} \to \dots \to A_0$ of
abelian groups, we define the \emph{$r$th derived system} to be the inverse
system of abelian groups $\left\{\mathrm{im}(A_{n+r} \to A_n)\right\}_{n \in \mathbb{Z}_{\geq
0}}$. 
\end{definition} 

\begin{definition} 
An inverse system of abelian groups $\dots \to A_n \to A_{n-1} \to \dots \to
A_0$ is \emph{eventually constant} if the maps $A_n \to A_{n-1}$ are
isomorphisms for $n \gg 0$. 
\end{definition} 

Since $\mathcal{C}$ is a stable $\infty$-category, there is a natural
\emph{spectrum}  of maps between any two objects $X, Y \in \mathcal{C}$, which
for the next result we write simply as $\hom(X, Y)$. 
\begin{proposition} \label{veryprozero}
The tower $\dots \to X_n \to X_{n-1} \to \dots$ in $\mathcal{C}$ belongs 
to $\townil_{\mathfrak{U}}(\mathcal{C})$ if and only if, for each $U \in
\mathfrak{U}$, the spectral sequence associated to the tower of spectra
\[ \dots \to \hom(U, X_n) \to \hom(U, X_{n-1}) \to \dots \to \hom(U, X_0),  \]
collapses to zero at a finite stage independent of $U$ (that is, the $r$th page is identically
zero for some $r$, independent of $U$). 
\end{proposition} 

\begin{proof} 
We consider first the case $\mathcal{C} = \sp$ and $\mathfrak{U} =
\left\{S^0\right\}$, and assume that we have a tower of spectra
$\left\{X_n\right\}_{n \in \mathbb{Z}_{\geq 0}}$, which we extend to $n < 0$ by taking $X_{-1} = X_{-2} =
\dots = 0$. 
In this case, we recall the definition 
of the spectral sequence associated to the tower. 
Let $F_i$ be the fiber of $X_i \to X_{i-1}$. 
One has an exact couple
\[ \xymatrix{
\bigoplus_i \pi_* X_i \ar[rr]^{\phi}  & &  \bigoplus_i \pi_* X_{i-1} \ar[ld]  \\
& \bigoplus_i \pi_* F_i \ar[lu]
},
\]
where $\phi$ is obtained as the direct sums of the maps $\pi_* X_{i} \to
\pi_* X_{i-1}$. 
The spectral sequence for the homotopy groups of $\varprojlim X_i$ is obtained by repeatedly deriving this exact couple. 
In particular, the successive derived  couples are of the form 
\[ 
\xymatrix{
\im \phi^{r-1}  \ar[rr]^{\phi}&&  \im \phi^{r-1} \ar[ld] \\
 & E_r^{\ast, \ast} \ar[lu]
},
\]
so that if $E_r^{\ast, \ast}$ is identically zero for some $r$, then 
$\phi$ necessarily (by exactness) induces an automorphism of $\im
\phi^{r-1}$. 
In other words, if we consider the inverse system 
$\{\pi_* X_i\}_{i \in \mathbb{Z}}$, it follows that the $r$th derived system is
constant, and thus necessarily zero, by considering indices below zero. 

For the converse, if $\left\{X_n\right\} \in
\townil_{\left\{S^0\right\}}(\sp)$, say all $r$-fold composites in the
inverse system are zero, then reversing the above argument shows that the exact
couple degenerates to zero at the $r+1$st stage: the $\im \phi^{r}$ terms are
zero. 

The case of a general $(\mathcal{C}, \mathfrak{U})$ now easily reduces to this,
because (as in the above argument), the point at which the spectral sequence
collapses and the number of successive composites needed to make all the maps
nullhomotopic are functions of each other. 

\end{proof}

Similarly, we can give a criterion for belonging to $\towg_{\mathfrak{U}}(
\mathcal{C})$.
First, however, we need to prove some lemmas about pro-systems of abelian groups.

\begin{lemma} 
\label{halfexact}
Consider a half-exact sequence of inverse systems
of abelian groups
\[  \left\{X_i \right\}  \to \left\{Y_i\right\} \to \left\{Z_i\right\}
. \]
Then:
\begin{enumerate}
\item  If $\left\{Y_i\right\}, \left\{Z_i\right\}$ have eventually constant
$r$th derived systems and $0 \to \left\{X_i\right\} \to \left\{Y_i\right\}$ is
exact, then $\left\{X_i\right\}$ has an eventually constant $r$th derived
system. 
\item If $\left\{X_i\right\}, \left\{Y_i\right\}$ have eventually constant $r$th
derived systems and $\left\{Y_i\right\} \to \left\{Z_i\right\} \to 0$ is exact, 
then $\left\{Z_i\right\}$ has an eventually constant $r$th derived system.  
\item 
If 
$$0 \to  \left\{X_i \right\}  \to \left\{Y_i\right\} \to \left\{Z_i\right\} \to
0,$$
is exact, and $\left\{X_i\right\}$ and $\left\{Z_i\right\}$ have eventually
constant $r$th derived systems, then $\left\{Y_i\right\}$ has an eventually
constant $2r$th derived system. 

\end{enumerate}

\end{lemma} 

\begin{proof}
We will prove the first and third of these assertions. The second is dual to
the first and can be proved similarly (or via an ``opposite category''
argument, since everything here works in an arbitrary abelian category). 
The three assertions state that the subcategory of the category of towers of
abelian groups consisting of those towers with an eventually constant $r$th
derived system for $r \gg 0$ is closed under 
finite limits and colimits, and extensions too. 

\begin{enumerate}[leftmargin=*] 
\item  Suppose $\left\{Y_i\right\}, \left\{Z_i\right\}$ have eventually
constant $r$th derived systems and $\left\{X_i\right\}$ is the kernel inverse
system of $\left\{Y_i\right\} \to \left\{Z_i\right\}$. 
In this case, we want to show that the $r$th derived system of $\mathrm{ker}( Y_i \to Z_i)$ 
is eventually constant. 

For $i \gg 0$, it follows that
\[ \phi\colon \mathrm{Im}( \phi^{r} \colon X_i \to X_{i-r}) \to \mathrm{Im}(
\phi^{r} \colon X_{i-1} \to X_{i-r-1}), \]
is injective, because the analog is true for $\left\{Y_i\right\}$. 

The harder step is to show that the map is surjective. Equivalently, for $j
\gg 0$, we must show that  any element of
$X_{j}$ that can be lifted up to $X_{j + r}$ can also be lifted up to
$X_{j+r+1}$ (not necessarily in a compatible manner). 
Fix $x_j \in X_j$ admitting a lift $x_{j+r} \in X_{j+r}$. Then the image
$y_{j+r}$ of $x_{j+r}$ in $Y_{j+r}$ lifts the image $y_j$ of $x_j$ in $Y_j$. It
follows that $y_j$, since it lifts $r$ times, is actually the image of an
element $y'_{j+r +1} \in Y_{j+r+1}$ which is ``permanent,'' i.e., which lifts
arbitrarily. The image $z'_{j+r+1} \in Z_{j+r}$ of $y'_{j+r+1}$ maps to zero in
$Z_{j}$, but $z'_{j+r+1}$ is also ``permanent,'' so it must itself vanish by
assumption. Thus $y'_{j+r+1} $ is the image of $x'_{j+r+1} \in X_{j+r+1}$ and
this lifts $x_j$.

\item[(3)]  Suppose $\left\{X_i\right\}, \left\{Z_i\right\}$ have eventually 
constant $r$th derived systems. Suppose $j \gg 0$ and suppose $y_j \in Y_j$ is in
the image of $Y_{j+2r}$. We need to show two things:
\begin{itemize}
\item  If $y_j \neq 0$, then the image of $y_j$ in $Y_{j-1}$ is not zero. 
\item $y_j$ can be lifted to $Y_{j+2r+1}$. 
\end{itemize}

For the first item, if the image $z_j$ of $y_j$ in $Z_{j}$ is not zero, then
the image $z_{j-1}$ of $z_j$ in $Z_{j-1}$ is also nonzero, because $z_j$ is
permanent. If $z_j = 0$, then $y_j$ comes from $X_j$, and we can apply the
same argument to $X_j$. 

Namely, choose $y_{j+2r}
\in Y_{j+2r}$ lifting $y_j$ and let $y_{j+r} \in Y_{j+r}$ be the image. The
image $z_{j+r} \in Z_{j+r}$ is a permanent element which, by hypothesis,
projects to zero in $Z_j$, so must be zero. In particular, $y_{j+r} \in
Y_{j+r}$ comes from $X_{j+r}$, which means that $y_j$ not only comes from $X_j$
but also that it is the image of a (nonzero) permanent element in $X_j$. The image of
this in $X_{j-1}$ thus cannot be zero. 
This completes the proof of the first item.

For the second item, the image $z_j \in Z_j$ of $y_j$ is a permanent element,
so it lifts uniquely to a permanent element $z_{j+2r+1}$.  Choose a lift
$y'_{j+2r+1}
\in Y_{j+2r+1}$ of $z_{j+2r+1}$. 
Let $y'_{j+2r}$ be the image of $y'_{j+2r+1}$ in $Y_{j+2r}$. Then $y_{j+2r}
-y'_{j+2r}$ maps to an element in $Z_{j+2r}$ which maps to zero in $Z_j$, and
thus maps to zero in $Z_{j+r}$. In particular, $y_{j+r} - y'_{j+r}$ belongs
to $X_{j+r}$; call it $\overline{x}_{j+r}$. Taking images in $Y_j$, we get
\[ y_j = \overline{x}_j + y'_j,  \]
where $\overline{x}_j$ is the image of $\overline{x}_{j+r}$ and $y'_j$ is the
image of $y'_{j+2r+1}$. But $\overline{x}_j$ is permanent, and $y'_j$ is the
image of something in $Y_{j+2r+1}$, so this completes the proof. 
\end{enumerate}

\end{proof}

\begin{lemma} 
\label{profivelem}
Let $\left\{A_i\right\}, \left\{B_i\right\}, \left\{C_i\right\}$ be three
pro-systems of graded abelian groups. 
Suppose that there is a long exact sequence of pro-systems
\[ \xymatrix{
\{A_i\} \ar[rr] & & \left\{B_i\right\} \ar[ld] \\
 & \left\{C_i\right\} \ar[lu]
},\]
where the map $\left\{C_i\right\} \to \left\{A_i\right\}$ lowers grading by $1$. 
Suppose the $r$th derived system of $\left\{A_i\right\}$ and
$\left\{B_i\right\}$ are eventually constant. Then the $2r$th derived system of
$\left\{C_i\right\}$ is eventually constant. 
\end{lemma} 
\begin{proof} 
More generally, suppose given an exact sequence of inverse systems of abelian
groups
\[ \left\{M_i\right\} \to \left\{N_i\right\}  \to \left\{P_i\right\} \to
\left\{Q_i\right\} \to \left\{R_i\right\}, \]
and suppose that the $r$th derived systems of the inverse systems
$\left\{M_i\right\}, \left\{N_i\right\}, \left\{Q_i\right\},
\left\{R_i\right\}$ are eventually constant. Then we claim that the $2r$th 
derived system of $\left\{P_i\right\}$ is eventually  constant. 
In fact, this follows by applying \Cref{halfexact} three times, and implies the present result.

\end{proof}

\begin{proposition} \label{veryproconst}
The tower $\dots \to X_n \to X_{n-1} \to \dots $ in $\mathcal{C}$ belongs to
$\towg_{\mathfrak{U}}(\mathcal{C})$ if and only if, for each $U \in
\mathfrak{U}$, the spectral sequence associated to the tower of spectra
\[ \dots \to \hom(U, X_n) \to \hom(U, X_{n-1}) \to \dots \to \hom(U, X_0),  \]
collapses at a finite stage  with a horizontal vanishing line independent of
$U$. In other words, there should exist $r, N$ such that $E_r^{s, \ast}  = 0$
for $s > N$, in the spectral sequence associated to the above tower for each
$U$.
\end{proposition} 
\begin{proof} 
Without loss of generality, suppose that $\mathfrak{U} = \left\{S^0\right\}$ and $\mathcal{C} = \sp$. 
Analysis with exact couples as in the proof of \Cref{veryprozero}
shows that the spectral sequence condition of the proposition holds if and only if there exists
$r, N$ such that the map
\[ \phi\colon \left( \im \phi^r\colon \pi_*( X_{s} )\to\pi_*( X_{s-r})\right)
\to \left( \im \phi^r\colon \pi_*( X_{s-1}) \to \pi_*(X_{s-r-1})\right)
\]
is an isomorphism, for all $s > N$. 
In other words, the $r$th derived system of the pro-system $\left\{\pi_*
X_n\right\}$ is \emph{eventually constant} (rather than constant at zero as in
the proof of \Cref{veryprozero}). 
Necessarily, the stable value of the pro-system must be $\pi_* \varprojlim X_i$.

 By \Cref{profivelem}, it follows that if the condition of
the proposition holds for the tower $\left\{X_i\right\}_{i \in
\mathbb{Z}_{\geq 0}}$, then it also holds
for the tower $\left\{\mathrm{cofib}( \varprojlim X_j\to X_i)\right\}_{i \in
\mathbb{Z}_{\geq 0}}$, since
it clearly holds for the constant tower with value $\varprojlim X_j$, and
conversely. It follows that we can reduce to the case where $\varprojlim X_j$
is contractible, which is precisely \Cref{veryprozero}.

\end{proof}

Observe that any object in $\towg_{\mathcal{C}}(\mathcal{C})$ defines a tower
yielding a 
\emph{constant} pro-object: this follows from the analogous assertion about
$\townil_{\mathcal{C}}(\mathcal{C})$. 

\begin{remark} Suppose $\mathfrak{U} = \mathcal{C}$ and $\mathcal{C}$ is
presentable. 
Then if $T$ is any spectrum and $\left\{X_n\right\} \in 
\towg_{\mathcal{C}}( \mathcal{C})$, the tower $\left\{T \wedge X_n\right\}$
also belongs to $\towg_{\mathcal{C}}( \mathcal{C})$, where we recall that
$\mathcal{C}$ is tensored over $\sp$. This is a consequence of
the fact that $\left\{X_n\right\} $ defines 
a constant pro-object in $\mathcal{C}$, so that
the natural map
\[ T \wedge \varprojlim X_i \to \varprojlim (T \wedge X_i ) , \]
is an equivalence. 
\end{remark}

\subsection{Vanishing in the descent spectral sequence}
Our goal is to show that  given $(\mathfrak{X}, \otop)$, then for any
$\mathcal{F} \in\qcoh ( \mathfrak{X})$, the $\mathrm{Tot}$ tower associated to
the given cosimplicial spectrum that ``computes'' $\Gamma( \mathfrak{X},
\mathcal{F})$ (i.e., based on a choice of affine, \'etale hypercovering of
$X$) has a
horizontal vanishing line in the homotopy spectral sequence. 

Let $\mathfrak{X} = (X, \otop)$ be an even periodic refinement of a
Deligne-Mumford stack $X$ equipped with a flat map $X \to M_{FG}$. 

\begin{theorem} \label{vanishingline} Suppose $X \to M_{FG}$ is
tame.\footnote{This is equivalent to the condition  that for
every geometric point $\spec k \to X$, the kernel of the automorphism group to the
automorphism group of the associated formal group has order prime to the
characteristic of $k$. We refer to \cite{AOV} for the general theory. In
\cite[Prop. 4.10]{MM}, it is shown that this hypothesis implies
that the global sections functor on the derived stack $\mathfrak{X}$ commutes
with filtered colimits. For example, representable maps are tame.}
There exists $s, N \in \mathbb{Z}_{>0}$ such that for any quasi-coherent sheaf
$\mathcal{F}$ on $\mathfrak{X}$, the descent spectral sequence for $\pi_*
\Gamma( \mathfrak{X}, \mathcal{F})$ has a horizontal vanishing line of height
 $N$ at the $s$th page.
 \end{theorem} 
\begin{proof} 
Note that it suffices to \emph{fix} $\mathcal{F}
$ and then prove the result for $\mathcal{F}$ specifically: if $s$ and $N$ could not be taken independently of
$\mathcal{F}$, taking appropriate wedges would provide a counterexample. 

Let $M$ be the least common multiple of the orders of all the automorphism
group schemes of geometric points of $X$.
Then, the stack $X[M^{-1}]$ is  itself tame, and so (as in \cite[Prop.
2.24]{MM}) has
bounded cohomology, and thus the $E_2$-page of the descent spectral sequence
has a horizontal vanishing line at $E_2$ itself after inverting some $M$.

We now work $p$-locally for a fixed ``bad'' prime $p$. 
In this case, all the spectra $\otop( \spec R)$ for \'etale maps $\spec R \to
X$ are $L_n$-local for some $n$; see the discussion at the beginning of
\cite[\S 4.2]{MM} and in particular \cite[Lemma 4.9]{MM}. 
The collection $\mathcal{C}_{\mathcal{F}} \subset \sp$ of spectra $T$ such that the
$\mathrm{Tot}$ tower for 
$\Gamma( \mathfrak{X} , \mathcal{F} \wedge T)$ belongs to $\towg_{\Gamma(
\mathfrak{X}, \otop)}( \md(\Gamma(
\mathfrak{X}, \otop))$ is a thick subcategory. 
Any thick subcategory of $L_n \sp$ that contains the smash powers of $E_n$
contains $L_n S^0$ by the Hopkins-Ravenel smash product theorem. It thus
suffices to show that the smash powers of $E_n$  belong to
$\mathcal{C}_{\mathcal{F}}$. Since $\mathcal{F}$ was arbitrary, we need to show
that the $\mathrm{Tot}$ tower for 
$\Gamma( \mathfrak{X} , \mathcal{F} \wedge E_n)$ belongs to 
$\towg_{\Gamma(
\mathfrak{X}, \otop)}( \md(\Gamma(
\mathfrak{X}, \otop)))$. 

However, as in \cite[Prop. 4.10]{MM}, since the stack $X \times_{M_{FG}} \spec \pi_0 E_n$
is tame, the descent spectral sequence
$\Gamma( \mathfrak{X} , \mathcal{F} \wedge E_n)$ already has a horizontal
vanishing line at $E_2$ (and therefore degenerates at a finite stage), because
the fiber product $X \times_{M_{FG}} \spec \pi_0 E_n$ is a (quasi-compact,
separated) \emph{tame} stack.  It follows that the statement of the proposition holds $p$-locally. 
\end{proof} 

\begin{remark}
The descent spectral sequence for $\pi_* \Gamma( \mathfrak{X}, \mathcal{F})$
will generally be very infinite at the $E_2$-page because of ``stackiness.'' The descent spectral sequence for
$KO$-theory, which is also the homotopy
fixed point spectral sequence for $KO \simeq KU^{h \mathbb{Z}/2}$, is displayed in \Cref{fig:KOss}. The class 
$\eta$ is not nilpotent in the $E_2$-page, until a $d_3$ kills
$\eta^3$. 
The same phenomenon occurs in the $\TMF$ spectral sequence, which is computed
in \cite{bauer} and \cite{konter}. It is a finiteness property of the
$E_n$-local stable homotopy category: the $E_n$-local Adams-Novikov spectral
sequence exhibits the same property \cite{HoveyS}. 
\end{remark}

\begin{figure}
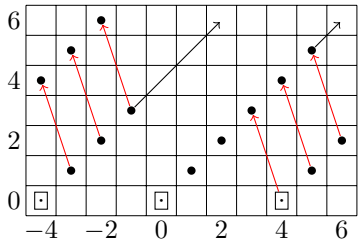

\begin{center}
\begin{sseq}{-4...6}{0...6}
\ssmoveto 0 0 
\ssdrop[boxed]{\cdot}
\ssmoveto 1 1 
\ssdropbull 
\ssmoveto 2 2 

\ssdropbull 
\ssmoveto 3 3 
\ssdropbull 
\ssmoveto 4 4 
\ssdropbull 
\ssmoveto 5 5 
\ssdropbull 

\ssmoveto 6 6 
\ssdrop{}
\ssstroke[arrowto]

\ssmoveto 4 0 
\ssdrop[boxed]{\cdot}
\ssarrow[color=red]{-1}{3}

\ssmoveto 5 1 
\ssdropbull 

\ssarrow[color=red]{-1}{3}
\ssmoveto 6 2 
\ssdropbull 

\ssarrow[color=red]{-1}{3}

\ssmoveto{-4}{0} 

\ssdrop[boxed]{\cdot}

\ssmove 1 1

\ssdropbull 

\ssmove 1 1
 
\ssdropbull

\ssmove 1 1
\ssdropbull

\ssmoveto 0 4

\ssmoveto{-4}{0} \ssmoveto{2}{6}
\ssdrop{}
\ssstroke[arrowto]

\ssmoveto {-4} 4 
\ssdropbull
\ssmove 1 1
\ssdropbull
\ssmove 1 1
\ssdropbull

\ssmoveto {-3} 1
\ssarrow[color=red]{-1}{3}

\ssmoveto {-2} 2
\ssarrow[color=red]{-1}{3}

\ssmoveto {-1} 3
\ssarrow[color=red]{-1}{3}

\end{sseq}
\end{center}
\caption{The descent (or homotopy fixed point) spectral sequence for $KO \simeq K^{h
\mathbb{Z}/2}$. Arrows in red indicate differentials, while arrows in black
indicate recurring patterns. Dots indicate copies of $\mathbb{Z}/2$, while
squares indicate copies of $\mathbb{Z}$.}
\label{fig:KOss}
\end{figure}

\section{Extension to stacks}

Fix a  \emph{regular} Deligne-Mumford stack $X$. 
Let $X \to M_{FG}$ be a flat morphism
and consider an even periodic refinement $\mathfrak{X} = ({X}, \otop)$ of this data. 
We have considered the $\infty$-category $\qcoh( \mathfrak{X})$ of
quasi-coherent sheaves on $\mathfrak{X}$, which has good finiteness
properties if $X \to M_{FG}$ is tame, as explored in \cite[\S 4]{MM}. 

In this section, we will define thick subcategories of $\qcoh( \mathfrak{X})$
associated to every closed substack of $X$ and prove one half of the thick
subcategory theorem. We will also illustrate how the other, more difficult, half can be deduced in
the special case of a quotient stack by a finite group.

\subsection{Definitions}

\newtheorem*{hyp}{Hypotheses}
We keep the previous notation and use the following hypotheses. 

\begin{hyp}
\begin{enumerate}
\item $X$ is a regular, separated, noetherian Deligne-Mumford stack.
\item 
$\mathfrak{X} = (X, \otop)$ is an even periodic refinement of a flat map $X \to
M_{FG}$.
\item The derived stack $\mathfrak{X}$ has the property that the global
sections functor establishes an equivalence $\md( \Gamma( \mathfrak{X}, \otop))
\simeq \mathrm{QCoh}( \mathfrak{X})$.
The main result of \cite{MM} implies that this holds if $X \to M_{FG}$ is
quasi-affine. 
\end{enumerate}
\end{hyp}
\begin{definition} 
Let $\qcoh^\omega( \mathfrak{X})$ be the subcategory of quasi-coherent sheaves
$\mathcal{F} \in \qcoh(\mathfrak{X})$, such that for each \'etale map $\spec R
\to X$, the $\otop( \spec R)$-module $\mathcal{F}( \spec R)$ is perfect. 
Thus, $\qcoh^\omega(\mathfrak{X})$ is the $\infty$-category of \emph{dualizable
objects} in $\qcoh( \mathfrak{X})$. 
Given $\mathcal{F} \in \qcoh^\omega( \mathfrak{X})$, the homotopy group sheaves
$\pi_i \mathcal{F}, i \in \mathbb{Z}$ define \emph{coherent} sheaves on $X$. 
\end{definition} 
\begin{remark}
Under taking global sections, $\qcoh^\omega(\mathfrak{X})$ then corresponds 
to the dualizable, or equivalently perfect, $\Gamma(\mathfrak{X},
\otop)$-modules. This follows because dualizability is a local condition (we
refer to \cite[\S 4.2.5]{higheralg} for the theory of duality in
$\infty$-categories), and
since the dualizable objects in a module category are precisely the perfect
modules. 
\end{remark}

Recall that there is a \emph{topological space} associated to $X$ in the sense of
\cite[Ch. 5]{LMB}. 
Fix a closed \emph{subset} $Z \subset X$: equivalently, this is an equivalence
class of closed substacks of $X$, where two closed substacks $Z, Z' \subset X$ are equivalent if
and only if there is a third substack $Z'' \subset X$ containing $Z, Z'$ as a
nilpotent thickening. 

\begin{definition}
Given a coherent sheaf $\mathcal{G}$ on $X$, we say that it is
\emph{supported} on $Z$ if there exists a closed substack of $X$ with $Z$ as
its underlying space on which $\mathcal{G}$ is supported. (This is a condition
on the fibers of $\mathcal{G}$ at field-valued points.) 

If we have a particular closed substack 
$Z \subset X$ in mind (and not simply an equivalence class), we will also say
that $\mathcal{G}$ is \emph{scheme-theoretically supported on $Z$} if
$\mathcal{G}$ is an $\mathcal{O}_Z$-module.
\end{definition}

\begin{definition} 
Let $Z$ be a subset of $X$ closed under specialization.
We define a thick subcategory $\qcoh_Z^\omega(\mathfrak{X}) \subset \qcoh^\omega(
\mathfrak{X})$ consisting of those $\mathcal{F} \in
\qcoh^\omega(\mathfrak{X})$  such that the homotopy group sheaves $\pi_i
\mathcal{F}$ are supported set-theoretically on $Z$. 
We thus get a map from specialization-closed subsets of $X$ to thick
subcategories of $\mo( \Gamma( \mathfrak{X}, \otop))$.
\end{definition} 

Using exact sequences, one sees that $\qcoh_Z^\omega( \mathfrak{X})$ is in fact
thick. Our first goal is to show that every thick subcategory of
$\qcoh^\omega(X)$ is of this form. 
As in \Cref{sec:affinecase}, it suffices to construct a sufficient collection
of ``residue fields.''
To do this, choose an \'etale surjection $\spec R \to X$. 
The $E_\infty$-ring $\otop( \spec R)$ fits into the setting of the previous
subsection, and we can construct homology theories $K( \mathfrak{p})_*$ on
$\md( \otop(\spec R))$ for
each $\mathfrak{p} \in \spec R$. 

Via pull-back, we can define these homology theories on $\qcoh(\mathfrak{X})$. 
In particular, it follows that given $\mathcal{F} \in \qcoh( \mathfrak{X})$, we
can define 
\[ K( \mathfrak{p})_* \mathcal{F} \stackrel{\mathrm{def}}{=} \pi_* (
\mathcal{F}( \spec R)_{\mathfrak{p}}/(x_1, \dots, x_r)),  \]
where $x_1, \dots, x_r \in R_{\mathfrak{p}}$ is a system of parameters. 
Every point of $X$ can be represented by a prime ideal $\mathfrak{p} \in \spec
R$ (the topological space of $X$ is the quotient of $\spec R$ under the
two maps $\spec R \times_X \spec R \rightrightarrows \spec R$), and as in the previous section, it follows that the support of $\mathcal{F}
\in \qcoh^\omega( \mathfrak{X})$ contains the point corresponding to
$\mathfrak{p} \in \spec R$ if and only if $K( \mathfrak{p})_* \mathcal{F} \neq
0$. 

\begin{proposition} 
Given $\mathcal{F} \in \qcoh( \mathfrak{X})$, $\mathcal{F} \simeq 0$ if and
only if $K( \mathfrak{p})_* \mathcal{F} = 0$ for each $\mathfrak{p} \in \spec
R$. 
\end{proposition} 

\begin{proof} 
This is a consequence of \Cref{bdecmdR}, since $\mathcal{F}$ is contractible if
and only if $\mathcal{F}( \spec R)$ is. 
\end{proof}

Now the entire thick subcategory argument
(reviewed in the previous section for $R$-modules) can be carried out in
$\qcoh^\omega( \mathfrak{X})$, or, one can appeal to the axiomatic framework
in \cite{axiomatic}. 
In particular, just as in \Cref{firstclaim}, one concludes:

\begin{proposition} 
If $\mathcal{F}, \mathcal{F}' \in \qcoh^\omega( \mathfrak{X})$ and $\supp
\mathcal{F} \subset \supp \mathcal{F}'$, then the thick subcategory generated
by $\mathcal{F}'$ contains $\mathcal{F}$.
\end{proposition}

\begin{corollary} \label{possiblethick}
Every thick subcategory of $\qcoh^\omega( \mathfrak{X})$ is of the form
$\qcoh^\omega_Z( \mathfrak{X})$ for some subset $Z \subset X$ closed under
specialization. 
\end{corollary} 

The primary goal of the rest of this paper will be to study when the different
$\qcoh^\omega_Z( \mathfrak{X})$ are distinct: that is, we would like to know
when we can realize
a given closed subset as the support of a sheaf on $\mathfrak{X}$. We will
answer this question, via \Cref{ourthick}, when $X \to M_{FG}$ is \emph{affine}. 

\subsection{The case of a quotient stack}
In this subsection, we consider a case of 
\Cref{ourthick}. 
Suppose the Deligne-Mumford stack $X$ is given by the quotient
of an affine scheme by a finite group, so 
\[ X = (\spec R_0)/G,  \]
where $|G| < \infty$, and $R_0$ is a regular noetherian ring. 
Consider a flat, affine map $X \to M_{FG}$ and an even periodic refinement
$(\mathfrak{X}, \otop)$ of this map. 
In this case, we have
 a $G$-action on the $E_\infty$-ring $R = \otop( \spec R_0)$, and 
 \[ \Gamma( \mathfrak{X}, \otop) \simeq R^{hG}.  \]
 This case was discussed in \cite{MM}, and it is quite general (for instance,
 it includes $\mell$ once any prime is inverted). 
By the main result of \cite{MM}, one has:

\begin{proposition} We have an equivalence
\[ \md( R^{hG}) \simeq \qcoh( \mathfrak{X}) \simeq \md(R)^{h G},  \]
\end{proposition}
Our goal is to show,  as claimed in \Cref{ourthick}, that every thick subcategory of $\md( R^{hG})$
should arise
from a $G$-invariant subset of $\spec R_0$ closed under specialization (i.e., a
union of closed substacks of
$X$).
We begin with some lemmas. 

\begin{lemma} 
\label{sshorline} Let $E_r^{s,t}$ be an upper half-plane spectral sequence of (not
necessarily commutative) algebras for $s \geq 0,
t \in \mathbb{Z}$. 
Suppose that $x \in E_2^{0, r}$ is an element. 
Suppose moreover: 
\begin{enumerate}
\item $x$ is central in $E_2^{\ast, \ast} $. 
\item $E_2^{s, t}$ is torsion for $s > 0$ and $t$ arbitrary.
\item The spectral sequence degenerates at a finite stage. 
\end{enumerate}
Then a power of $x$ survives to $E_\infty^{0, r}$. 
\end{lemma} 
\begin{proof} 
In fact, suppose $d_2(x)$ is $N$-torsion;
then this implies that $x^N$ survives to $E_3$. Repeating, it follows that a
sufficiently divisible enough power of $x$ survives to $E_4, E_5$, and so forth to
any finite stage. Since the spectral sequence stops at a finite stage, it
follows that a high power of $x$ survives the spectral sequence. 
\end{proof} 

\begin{lemma} \label{lemnormgal}
Let $x \in R_0^G$. Then a sufficiently divisible power of $x$ is in the image of
$\pi_0 R^{hG} \to \pi_0 R$. 
\end{lemma} 

\begin{proof} 
To see this, consider the homotopy fixed point spectral sequence
\[ H^i( G; \pi_j R) \implies \pi_{j-i} R^{hG}.  \]
By assumption, $x$ defines an element of $E_2^{0, 0}$. This spectral
sequence has the two key properties of \Cref{sshorline}. 
Above the $s = 0$ line, everything is torsion: in fact, annihilated by
$|G|$. 
Moreover, the spectral sequence degenerates at a finite stage (in fact,  with a horizontal
vanishing line), thanks to \Cref{vanishingline}. 
This is enough to imply the lemma by \Cref{sshorline}. \end{proof} 

\begin{remark} 
If $x$ is invertible, one may prove this using the \emph{norm map} $ \mathfrak{gl}_1( R) \to
\mathfrak{gl}_1(R^{hG}) \simeq \mathfrak{gl}_1(R)^{hG}$.

\end{remark} 

\begin{proposition} \label{galthickc}
The thick subcategories of $\mo( \Gamma( \mathfrak{X}, \otop))$ are in
bijection with the $G$-invariant subsets of $\spec R_0$ closed under
specialization. 
\end{proposition} 

\begin{proof} 
By \Cref{possiblethick}, 
the remaining step is to show that given a $G$-invariant closed subset $Z \subset
\spec R_0$, we can find a  perfect $ \Gamma( \mathfrak{X}, \otop)$-module with
support exactly $Z$. 
We would like to imitate the construction used to prove \Cref{thickaffine}, although the problem is
that the $x_i$ need not live inside $\pi_* \Gamma( \mathfrak{X}, \otop)  =
\pi_* R^{hG}$. We
can get around this as follows. 

Let $Z$ correspond to the $G$-invariant radical ideal $I
\subset R_0$. We can find an ideal $J \subset I$ such that $\mathrm{rad}(J) =
I$ and such that $J$ is itself generated by $G$-invariant elements. 
Indeed, for each prime
$\mathfrak{p}$ that fails to contain $I$, 
we observe that $g \mathfrak{p}$ also fails to contain $I$ for each $g \in G$. 
Therefore, by prime avoidance \cite[Lemma 3.3]{eisenbud}, choose an element $x \in I \setminus
\bigcup_{g \in G} (g\mathfrak{p})$. Then
consider its \emph{norm} $N_G x = \prod_{g \in G} gx$, 
which is in $I$ and not in $\mathfrak{p}$. 
Taking norms such as these, we can choose $G$-invariant elements of $I$ such that any prime
that fails to contain  $I$ fails to contain one of these elements. 

Therefore, we choose $G$-invariant elements $\left\{x_1, \dots, x_n\right\}
\subset R_0^G$ such that they cut out the closed \emph{subset} $Z \subset \spec
R_0$. We would like to take as our $\Gamma( \mathfrak{X}, \otop)$-module $M$
the iterated quotient $\Gamma( \mathfrak{X}, \otop)/(x_1, \dots, x_n)$, except
that the $x_i$ do not necessarily belong to $\pi_0 \Gamma( \mathfrak{X},
\otop)$. 
However, by \Cref{lemnormgal}, after raising the $x_i$ to a suitable power, we
can arrange that they do belong to $\pi_0 \Gamma(\mathfrak{X}, \otop)$.
Then, the quotient $R^{hG}/(x_1, \dots, x_n)$ is the
desired $R^{hG}$-module. 

\end{proof}

\begin{corollary} 
Suppose $R$ is an $E_\infty$-ring with an action of a finite group $G$. Suppose
that
\begin{enumerate}
\item $R$ is even periodic with $\pi_0 R$ regular.  
\item $R^{hG} \to R$ is a faithful $G$-Galois extension (in the sense of
\cite{rognes}). 
\end{enumerate}
Then the thick subcategories of $\md^\omega(R^{hG})$ are in bijection with the
$G$-invariant subsets of $\spec \pi_0 R $ closed under specialization. 
\end{corollary} 
\begin{proof} 
This result is not a corollary of the previous ones, but rather of the proof.
Since $R^{hG} \to R$ is faithful $G$-Galois, we have an equivalence of
$\infty$-categories $\md(R^{hG}) \simeq \md(R)^{hG}$. This is essentially
Galois descent, and has been observed independently by Gepner-Lawson and
Meier; see for example \cite[Th. 9.5]{galois}. 

Moreover, since $R$ is even
periodic with $\pi_0 R$ regular, we can use the residue
field construction of \Cref{k(p)} to produce multiplicative homology theories on
$\md(R^{hG})$ satisfying K\"unneth isomorphisms that detect all nonzero objects. 
This gives one half of the classification of thick subcategories. 

Using the same construction as above, 
one sees that the argument proving \Cref{galthickc} applies here too,
provided that \Cref{lemnormgal} applies. 
But in fact \Cref{lemnormgal} is valid for any faithful $G$-Galois extension. 
As shown in \cite{galois}, the map $R^{hG} \to R$ is ``descendable'' in the sense
of \S 3-4 of that paper. 
One can use the analogous properties of the homotopy fixed point
(or descent) spectral sequence (degeneration at a finite stage with a horizontal vanishing
line, so that \Cref{sshorline} applies) which are valid for the homotopy fixed point spectral sequence for any
faithful $G$-Galois extension; see \cite[\S 4]{galois}. 

\end{proof} 

\section{The general case}

In this section, we complete the proof of \Cref{ourthick}. 
Throughout this section, we fix the even periodic refinement $\mathfrak{X} =
({X}, \otop)$ of the affine map $X \to M_{FG}$, where $X$ is a regular,
noetherian, and separated Deligne-Mumford stack. 
Our goal is now to produce sufficiently many perfect $\Gamma( \mathfrak{X},
\otop)$-modules to see that any closed substack of $X$ can be realized as the
support.
Since $X$ need not be a quotient stack by a finite group, we will need
a different approach from the previous section.  We will use the theory of
graded Hopf algebroids (see, e.g., \cite{green}). 

\subsection{An abstract periodicity theorem}

The basic step in producing $\Gamma( \mathfrak{X}, \otop)$-modules is the following analog of the periodicity theorem in our setting.
\begin{theorem} \label{ourperiod}
Let $\mathcal{F} \in \qcoh^{\omega}( \mathfrak{X})$  have the property that
the homotopy group sheaves of $\mathcal{F}\wedge \mathbb{D}\mathcal{F}$ are scheme-theoretically supported on a closed substack $Z \subset X$. Given a
section $s \in H^0( Z, \omega^k)$, there exists a self-map
\[ \Sigma^{nk}\mathcal{F} \to \mathcal{F},  \]
for some $n$, 
whose map on homotopy group sheaves is given by multiplication by $s^{\wedge
n}$. 
\end{theorem} 
\begin{proof} 
In fact, we consider the endomorphism ring $\mathrm{End}( \mathcal{F})$,
which is an $A_\infty$-algebra
internal to the category $\qcoh^{\omega}(\mathfrak{X})$. Note first that
$\mathrm{End}(\mathcal{F})$ has its homotopy group sheaves 
supported on the closed substack $Z$. The homotopy group sheaves
$\mathrm{End}_*(\mathcal{F})$ of
$\mathrm{End}( \mathcal{F})$ form a sheaf of graded associative algebras on
$Z$, and we get a map
of quasi-coherent sheaves on $Z$ (or $X$)
\[ \bigoplus_{k = - \infty}^\infty \omega_Z^{\otimes k}  \to \mathrm{End}_*( \mathcal{F}).  \]
This is obtained from the natural map
\[  \bigoplus_{k = - \infty}^\infty \omega^{\otimes k}  \to \mathrm{End}_*( \mathcal{F}),  \]
which has the property that it factors through the base-change to $Z$. 
Note in particular that this map is central. 
In particular,  the section $s \in H^0( Z, \omega^k)$ defines a central element of
$H^0( X, \pi_k \mathrm{End}( \mathcal{F}))$: more precisely, a \emph{central} element in
the $E_2$-page of
the spectral sequence 
\[ E_2^{s,t} = H^i( X, \pi_j \mathrm{End}( \mathcal{F})) \implies \pi_{j-i}
\Gamma(\mathfrak{X}, \mathrm{End}(\mathcal{F})).  \]
In this spectral sequence, everything above the horizontal line $s = 0$ is torsion, as the
rationalization $X_{\mathbb{Q}}$ is the quotient of an affine scheme by
$\mathbb{G}_m$ in view of the affine map $X_{\mathbb{Q}} \to
(M_{FG})_{\mathbb{Q}} \simeq B \mathbb{G}_m$.
In particular, the presentation $X_{\mathbb{Q}} = (\mathrm{affine})/\mathbb{G}_m$ implies
that $X_{\mathbb{Q}}$ has no higher sheaf cohomology. 
Therefore, 
it follows by
\Cref{vanishingline} and \Cref{sshorline} that a power of $s$ survives the spectral sequence and defines a
global endomorphism of $\mathcal{F}$ as desired. 
\end{proof} 

This result \emph{almost} reduces our work to pure algebra. 
The situation becomes slightly tricky, though, because while the
set-theoretic support is well-behaved in cofiber sequences, the
\emph{scheme-theoretic} support (which is what intervenes in \Cref{ourperiod}) is less
so. We now note further consequences of \Cref{ourperiod} that will be used in
the sequel.

\begin{lemma} 
\label{froblift}
Let $Y \subset Y'$ be a nilpotent thickening of Artin stacks, and let
$\mathcal{L} \in \pic( {Y}')$. Then any  \emph{torsion} section $s \in H^0( Y,
\mathcal{L})$ has the property that some power of $s$ extends over $Y'$. 
\end{lemma} 
\begin{proof} 
It suffices to consider a square-zero thickening $Y \subset Y'$, defined by a
square-zero sheaf of ideals $\mathcal{I}$ on $Y'$. Then we have an exact
sequence of sheaves
on $Y'$
\[ 0 \to \mathcal{I}\mathcal{L} \to \mathcal{L} \to \mathcal{L}
\otimes_{\mathcal{O}_{Y'}} \mathcal{O}_Y \to 0,  \]
and we consider a section $s$ of the last term. The obstruction to its lifting
is given by the coboundary $\delta s \in H^1( \mathcal{I} \mathcal{L})$. The
coboundary has the property $\delta( s^N) = N s^{N-1} \delta(s)$ (in a natural sense,
given the operation of $\mathcal{O}_{Y'}/\mathcal{I}$ on $\mathcal{I}$), which
vanishes for $N$ highly divisible by assumption. 
\end{proof} 

\begin{corollary} 
\label{cofibertorsion}
Let $\mathcal{F} \in \qcoh^\omega(\mathfrak{X})$. Suppose $\mathcal{F} $ is
supported (scheme-theoretically) on a closed substack $Z \subset X$. Suppose
$Z' \subset X$ is another closed substack with the same underlying 
set as $Z$
and $s \in H^0( Z', \omega^i)$ is a torsion section. Let $Z''$ be the closed
substack of $Z'$ cut out  by $s$. Then there exists $\mathcal{F}' \in
\qcoh^\omega( \mathfrak{X})$ whose set-theoretic support is precisely $Z''$.
\end{corollary} 
\begin{proof} 
By \Cref{froblift}, we may assume, after raising $s$ to an appropriate power,
that $Z' = Z$ and $s$ actually is a section over $Z$. In this case, we use
\Cref{ourperiod} to produce a self-map of $\mathcal{F}$ which induces
multiplication by some tensor power of $s$ on homotopy group sheaves. The
cofiber of this map can be taken to be $\mathcal{F}'$.
\end{proof}

\subsection{The algebraic setup; the rational piece}
\label{sec:algebraicsetup}

Let $Z \subset X$ be a closed substack. In this subsection, we will begin the
algebraic preliminaries in showing that there exists an object in 
$\qcoh^\omega(\mathfrak{X})$ set-theoretically supported on $Z$. 

\begin{definition} 

Recall (e.g., from \cite{goerssqc}) the covers 
\[ M_{FG}^{\mathrm{coord}, n} \to M_{FG},  \]
where 
$M_{FG}^{\mathrm{coord} , n}$ is the moduli stack of formal groups together
with a coordinate to degree $n$. 
Each of these covers is a torsor for the group $G$ that acts on coordinates to
degree $n$ (i.e., automorphisms of $\spec \mathbb{Z}[x]/x^{n+1}$). The group
$G$ has a map (which admits a splitting)
\[ G \twoheadrightarrow \mathbb{G}_m,  \]
by contemplating the action on the Lie algebra, and the kernel $H \subset G$ is
an iterated extension of copies of $\mathbb{G}_a$, 
This property of the group $G$ will become crucial below. 

\end{definition}

The inverse limit of the moduli stacks $M_{FG}^{\mathrm{coord}, n}$
parametrizes formal groups together with a coordinate: equivalently, formal
group laws. The inverse limit is thus the spectrum of the Lazard ring $L$. 
Since $X \times_{M_{FG}} \spec L$ is affine by hypothesis, one gets: 

\begin{proposition} 
$X \times_{M_{FG}} M_{FG}^{\mathrm{coord}, n}$ is affine for $n \gg 0$. 
\end{proposition} 
\begin{proof} 
Consider the tower of Deligne-Mumford stacks $ X^{(n)} = X \times_{M_{FG}^{\mathrm{coord},
n}} M_{FG}$ as $n$ varies. 
For $n > 0$, the successive maps in the tower are $\mathbb{G}_a$-torsors
and in particular are affine morphisms.

Moreover, the inverse limit of this tower, given by $X^{(\infty)}
\stackrel{\mathrm{def}}{=}X \times_{M_{FG}} \spec L$ is an affine scheme,
by hypothesis, so we want to claim that some term in the tower is itself an
affine scheme. For this, we argue first that for $n \gg 0$, 
the Deligne-Mumford stacks 
$X^{(n)}$ are \emph{algebraic spaces,} or equivalently, by \cite[Corollary
8.1.1]{LMB}, that they have no nontrivial automorphisms. 
In fact, consider the diagonal maps
\[ X^{(n)} \to X^{(n)} \times X^{(n)},  \]
for each $n$. These fit into a tower of cartesian squares as $n \to \infty$,
since for any morphism of stacks $Y_1 \to Y_2$, one has a cartesian square
\[ \xymatrix{
Y_1 \ar[d] \ar[r] &  Y_1 \times Y_1 \ar[d] \\
Y_2 \ar[r] &  Y_2 \times Y_2
}.\]
Since in the inverse limit, the map 
$X^{(\infty)} \to X^{(\infty)} \times X^{(\infty)}$ is a closed immersion, it
follows by \cite[Prop. B.3]{Rydh} that $X^{(n)} \to X^{(n)} \times X^{(n)}$ is
a closed immersion for $n \gg 0$. Thus, $X^{(n)}$ is an algebraic space for $n
\gg 0$.

Now, we can apply to the general theory 
of inverse limits of towers of
algebraic spaces under affine morphisms: by 
\cite[Tag 07SQ]{stacks-project}
if the inverse limit is affine, then
some term in the tower (and thus everything above it) must be affine, to
conclude. 
\end{proof} 
Fix one such $n$. Then we get a quotient stack presentation for $X$ as the
quotient of some affine scheme $\spec R = X \times_{M_{FG}}
M_{FG}^{\mathrm{coord}, n}$ by an action of the algebraic group
$G$. 
In particular, if $\mathcal{O}(G)$ denotes the ring of functions on $G$, then we get a
presentation for $X$ via a \emph{Hopf algebroid},
\begin{equation} \label{hopfalgebroid} \Gamma\colon \quad  R \rightrightarrows
R \otimes \mathcal{O}(G) \triplearrows \dots,  \end{equation}
\begin{remark}
Although we do not need this, these covers arise from certain ring spectra
$X(n)$. 
This Hopf algebroid can be realized in homotopy via the cobar construction
\[ \Gamma( \mathfrak{X} , \otop) \wedge X(n) \rightrightarrows  \Gamma(
\mathfrak{X} , \otop) \wedge X(n) \wedge X(n) \triplearrows \dots.\]
\end{remark}

Consider the setup above. 
If we forget the $\mathbb{G}_m$-action  but
remember the associated grading, the result is a \emph{graded} Hopf algebroid which presents the stack $X$. 
A given closed substack of $X \simeq \st(\Gamma( \mathfrak{X}, \otop))$ corresponds
to an invariant homogeneous ideal $I \subset R_*$. 

Our strategy will be, first, to choose \emph{globally invariant} elements $x_1, \dots, x_r$ that generate
$I$ \emph{rationally}, and which exist in homotopy in a very strong sense. Here
we use the fact that the stack $X$ is (up to $\mathbb{G}_m$-action) already
affine once we rationalize. After we do this, we need to add in more generators
to avoid introducing unnecessary irreducible components of the support. In the 
torsion case, however, the distinction between the set-theoretic and
scheme-theoretic support simplifies thanks to the Frobenius.

\begin{lemma} 
In the above setup, there exist invariant homogeneous elements $x_1, \dots, x_r$ in the Hopf
algebroid $\Gamma$ (from \eqref{hopfalgebroid}) that generate $I$ rationally.  
In the language of stacks, there are sections $x_i \in H^0( X, \omega^{k_i})$
which cut out the closed substack $Z$ rationally. 
\end{lemma} 
\begin{proof} 
We will prove this 
using a \emph{different} presentation from \eqref{hopfalgebroid}.
On rationalizations,  $X_{\mathbb{Q}} \to (M_{FG})_{\mathbb{Q}} \simeq B
\mathbb{G}_m$ is affine, so $X_{\mathbb{Q}}$ is the $\mathbb{G}_m$-quotient of
an affine scheme $ \spec C$ with a $\mathbb{G}_m$-action (i.e., grading). Now
a closed substack of $X_{\mathbb{Q}} \simeq (\spec C)/\mathbb{G}_m$ is defined
by a $\mathbb{G}_m$-invariant ideal of $C$, or a homogeneous ideal of $C$. We
can take $x_1, \dots, x_r $ as homogeneous elements of $C$ (i.e., sections of
$H^0( X_{\mathbb{Q}}, \omega^{k_i})$) which generate this homogeneous ideal.
Multiplying by a highly divisible integer, we may assume they extend to
sections over $X$. 
\end{proof}

\begin{proposition} 
\label{rationalexistence}
Given the closed substack $Z \subset X$,
there exists $\mathcal{F}' \in \qcoh^\omega(\mathfrak{X})$ such that $\supp
\mathcal{F}' \supset Z$ and $(\supp \mathcal{F}')_{\mathbb{Q}} =
Z_{\mathbb{Q}}$.
\end{proposition} 
\begin{proof} 
Choose sections $x_1, \dots, x_r \in H^0( X, \omega^{\otimes \bullet})$ such
that the closed substack of $X$ cut out by the $\left\{x_i\right\}$ is equal,
rationally, to $Z$. After multiplying the $x_i$ by a sufficiently divisible
integer, we may assume the $x_i$ vanish along $Z$ as well. After raising the
$x_i$ to a sufficiently divisible power, we may assume (by
\Cref{vanishingline} and \Cref{sshorline})
that each $x_i$ survives to an element in $\pi_*( \Gamma( \mathfrak{X},
\otop))$.
Then, we can take $\mathcal{F}' = \otop/x_1 \wedge \dots \wedge \otop/x_r$.
\end{proof}

\subsection{The torsion piece}
In this subsection, we prove some algebraic lemmas needed to handle the torsion. 
Throughout, let $B$ be a base ring, assumed noetherian. 
Let $G$ be an algebraic group over $B$ with the property that $G$ fits into 
an exact sequence
of group schemes
\[ 1 \to H \to G \to \mathbb{G}_m \to 1,  \]
where the map $G \to \mathbb{G}_m$ has a section (so that $G$ is a semidirect
product). 
Suppose $H$ has a finite filtration with successive quotients isomorphic to
$\mathbb{G}_a$. 
Observe that on any $G$-quotient, there is a natural line bundle $\omega$
obtained from the map $G \to \mathbb{G}_m$, the standard one-dimensional representation
of $\mathbb{G}_m$, and the Borel construction. 
Throughout, a \emph{representation} of an algebraic group (always over the base
ring $B$) will refer to a $B$-module together with a coaction of the associated
Hopf algebra. Given a representation, the \emph{fixed points}  consist of the
primitive vectors under the coaction map. See \cite{waterhouse} for a
discussion over fields. 

\begin{lemma} 
Let $M$ be a $\mathbb{G}_a$-representation. Then if $M \neq 0$, $M^{\mathbb{G}_a} \neq 0$. 
\end{lemma} 
\begin{proof} 
Fix $m \neq 0 $ in $M$. Consider the coaction map $\psi\colon M \to M
\otimes_{B} B[t]$, and suppose:
\begin{equation} \label{ga1} \psi(m) = \sum_{i=0}^\infty t^i \otimes m_i, \quad m_i =0 \text{ for } i
\gg 0.  \end{equation}
Recall that $m_0 = m$, and in particular $\psi$ is injective. 
Let $n$ be maximal such that $m_n \neq 0$. Then $m_n$ is
$\mathbb{G}_a$-invariant. 
In fact, 
we get an equality from coassociativity,
\[ \sum_{i = 0}^n t^i \otimes \psi(m_i) = \sum_{i=0}^n \sum_{j=0}^i
\binom{i}{j} t^j \otimes t^{i-j} \otimes m_i , \]
and comparing terms of $t^n \otimes (\dots)$ shows that $\psi(m_n) = 1 \otimes
m_n$. 
\end{proof}

\begin{lemma} 
\label{ganilp}
Let $A$ be a $B$-algebra with 
an action of the algebraic group $H$. 
Let $I \subset A$ be a $H$-invariant
torsion ideal. Suppose
$I^H$ consists of nilpotent elements. Then $I$ is nilpotent. 
\end{lemma}
\begin{proof} 
We first consider the case of $H = \mathbb{G}_a$. 
Let $\psi\colon A \to A \otimes_B B[t]$ be the coaction map. 
Choose $x \in I$, and write $\psi(x) = \sum_{i=0}^n t^i \otimes x_i$ for the
$x_i \in I$ and for some $n$. By the proof of \eqref{ga1}, it follows that $x_n$ is
$H$-invariant and therefore, by assumption, nilpotent. 
Choose $N$ highly divisible and so large that $x_n^N = 0$. 

In this case, it follows that 
$$\psi(x^N) = \left(\sum_{i=0}^{n} t^i \otimes x_i \right)^N
= \left(\sum_{i=0}^{n-1} t^i \otimes x_i \right)^N,$$
because $N$ is highly divisible and $x_n$ is torsion and nilpotent. 
(In general, if $a$ is torsion and nilpotent, then $(a+b)^N = b^N$ for $N$
sufficiently highly divisible.) 
Therefore, when one expands $\psi(x^N) = \sum_{j=0}^\infty t^j \otimes n_j$,
the largest $j$ that appears is $j = N(n-1)$, and that term is $t^{N(n-1)}
\otimes x_{n-1}^N$. It follows that $x_{n-1}^N$ is $H$-invariant and therefore
nilpotent. Continuing in this way, we can work our way down to conclude that
all the $x_i$, and in particular $x_0 = x$, are nilpotent. 

In general, if $H$ is not assumed to be isomorphic to $\mathbb{G}_a$, choose an
exact sequence
\[ 1 \to H' \to H \to \mathbb{G}_a\to 1,  \]
and assume inductively that the lemma is valid for $H'$. In particular, if $I$
is not nilpotent, it
follows that $I^{H'} \subset A^{H'}$ contains a non-nilpotent element $x$. The group
$\mathbb{G}_a$ acts on $A^{H'}$ and $I^{H'}$, and in particular
$(I^{H'})^{\mathbb{G}_a}$ must contain a non-nilpotent element by the case of 
the lemma already proved. But then $I^{H}$ contains a non-nilpotent element, a
contradiction. 
\end{proof} 

 \begin{proposition} \label{algfact} \label{finallem} 
Let $Y$ be an Artin stack obtained as $Y \simeq \spec R/G$, where $R$ is a
noetherian ring. Then 
given a closed substack $T \subset Y$ such that $T_{\mathbb{Q}} =
Y_{\mathbb{Q}}$, there exists a sequence of closed
substacks
\[ Y \supset Y_1 \supset Y_2 \supset \dots \supset Y_r  \supset T,\]
such that:
\begin{enumerate}
\item There exists an element $y_i \in H^0( Y_i, \omega^{k_i} )$ such that
$Y_{i+1}$ is the zero locus of $y_i$. 
\item $Y_r$ is a nilpotent thickening of $T$. 
\end{enumerate}
\end{proposition}

\begin{proof}
The stack $Y$ is represented by a Hopf algebroid obtained by the $G$-action on
$\spec R$. The closed substack $T \subset Y$ corresponds to a $G$-invariant
ideal $I \subset R$ with the property that $I \otimes_{\mathbb{Z}} \mathbb{Q}
= 0$. It suffices to show that there exist homogeneous
 elements $y_1, \dots,
y_r \in R$ such that
\begin{enumerate}
\item The image of $y_i$  in $H$-invariant in $R/ (y_1, \dots, y_{i-1})$ (which
is inductively a $G$-representation by this assumption). 
\item $(y_1, \dots, y_r)$ contains a power of $I$. 
\end{enumerate}

To do this, note first that we may assume $I$ non-nilpotent. In this case,
\Cref{ganilp} gives us a non-nilpotent $H$-invariant element $y_1 \in I$,
which we can assume homogeneous. We can now form the
quotient $R/(y_1)$, which defines a 
proper closed $G$-invariant subscheme of  $\spec R$, or equivalently a proper
closed substack $Y_1 \subset Y$, which contains $T$. Now, apply \Cref{ganilp} again
to the $H$-action on $R/(y_1)$ and the image of $I$ in here, and continue to get the descending sequence of
substacks and the $y_i$. The process stops once the image of $I$ in $R/(y_1,
\dots, y_r)$ is nilpotent, at which point we have gotten down to a nilpotent
thickening of $T$. 
\end{proof}

\subsection{Proof of the main theorem}
We now restate and prove the main theorem from the introduction.
\begin{theorem} 
The construction $Z \mapsto \qcoh_Z^\omega(\mathfrak{X})$ establishes a
bijection between specialization-closed subsets of the underlying space of $X$
and thick subcategories of $\md^\omega( \Gamma( \mathfrak{X}, \otop)) \simeq
\qcoh^\omega( \mathfrak{X})$.
\end{theorem} 

\begin{proof}

It thus suffices to show that, given a closed substack $Z \subset X$, there
exists $\mathcal{F} \in \qcoh^\omega(\mathfrak{X})$ such that the
(set-theoretic) support of $\mathcal{F}$ is precisely $Z$.

By \Cref{rationalexistence}, there exists $\mathcal{F}' \in \qcoh^\omega( \mathfrak{X})$
such that the homotopy group sheaves of $\mathcal{F}'$ are supported
scheme-theoretically on a closed substack $Z'$ of $X$ with $Z' \supset Z$ and
$Z'_{\mathbb{Q}}  = Z_{\mathbb{Q}}$.
Moreover, by \Cref{finallem} (which is applicable in view of the discussion
in \Cref{sec:algebraicsetup}), there exists a descending sequence
of closed substacks
\[ Z' = Z_{1} \supset Z_{2}\supset \dots \supset Z_m \supset Z,  \]
and torsion sections $\overline{x}_i \in H^0( Z_i, \omega^{k_i})$ for $1 \leq i \leq
m$, such that:
\begin{itemize}
\item  $Z_{i+1}$ is the zero locus of $\overline{x}_i$ on $Z_i$. 
\item $Z_m$ is a nilpotent thickening of $Z$. 
\end{itemize}

We claim that, for each $i =  1, 2, \dots, m$, there exists $\mathcal{F}_i \in
\qcoh^\omega( \mathfrak{X})$ such that the set-theoretic support of
$\mathcal{F}_i$ is precisely $Z_i$. We prove this by induction on $i$. For $i =
1$, we can take $\mathcal{F}'$.
If we have proved the assertion for $i$, then the assertion follows for $i + 1$
by \Cref{cofibertorsion}. 
Taking $i = m-1$, we have proved our result. 
\end{proof} 
\newtheorem*{question}{Question}
\begin{question} 
Does $X$ need to be regular for the results of this paper to hold? 
\end{question} 

\bibliographystyle{alpha}
\bibliography{thick}

\end{document}